\documentclass[review]{elsarticle}
\usepackage[abs]{overpic}
\usepackage{graphicx}
\usepackage{subfigure}
\usepackage{subcaption}
\usepackage{stmaryrd}
\usepackage{amsfonts,color}
\usepackage{amsmath,amsthm,amssymb}
\usepackage{epsfig}
\usepackage{epstopdf}
\usepackage{psfrag}
\usepackage{paralist}
\usepackage{titlesec}

\usepackage{algorithm,algpseudocode}
\usepackage[most]{tcolorbox}
\tcbuselibrary{breakable}

\usepackage[%
    a4paper,       
    left=2.5cm,    
    right=2.5cm,   
    top=2.5cm,     
    bottom=2.5cm,  
    includehead,   
    includefoot    
]{geometry}

\theoremstyle{definition}
\newtheorem{lemma}{Lemma}[section]
\newtheorem{theorem}{Theorem}[section]

\newtheorem{remark}{Remark}[section]
\newtheorem{example}{Example}[section]

\usepackage{lineno,hyperref}

\journal{Journal of \LaTeX\ Templates}

\graphicspath{{../figsD/}}

\begin{document}
	
	\begin{frontmatter}
		
		\title{An asymptotically compatible unfitted finite element methods for nonlocal elliptic Interfaces:  local limits and sharp error estimates \tnoteref{mytitlenote}
		}

\author[mymainaddress]{Haixia Dong}
\author[mymainaddress]{Ziqing Xie}
\author[mythirdaddress]{ Jiwei Zhang  \corref{mycorrespondingauthor}}
\cortext[mycorrespondingauthor]{Corresponding author}
\ead{jiweizhang@whu.edu.cn}
\address[mymainaddress]{MOE-LCSM, School of Mathematics and Statistics, Hunan Normal University, Changsha, Hunan 410081, P. R. China.}
\address[mythirdaddress]{Corresponding author. School of Mathematics and Statistics, and Hubei Key Laboratory of Computational Science, Wuhan University, Wuhan 430072, China}

\begin{abstract}
	This paper  presents the development and analysis of an asymptotically compatible (AC) unfitted finite element method for one-dimensional nonlocal elliptic interface problems. The proposed  method achieves optimal error estimates through three principal contributions: 
(i) an extended maximum principle, coupled with an asymptotic consistency analysis of the flux operator, which establishes second-order convergence of nonlocal solutions to their local counterparts in the maximum norm;
(ii) a Nitsche-type formulation that directly incorporates nonlocal jump conditions into the weak form, enabling high accuracy without body-fitted meshes; and 
(iii) a rigorous proof of optimal convergence rates in both the energy and $L^2$ norms via the nonlocal maximum principle, flux consistency, and a newly derived nonlocal Poincaré inequality.
Numerical experiments confirm the theoretical findings and demonstrate the robustness and efficiency of the proposed approach, thereby providing a foundation for extensions to higher dimensions.
\end{abstract}
		
		
\end{frontmatter}
	
	
\section{Introduction}	
	   Nonlocal models (NLMs) have become indispensable tools for simulating long-range interactions and multiscale material behavior, with  applications spanning  fracture mechanics  \cite{MR1727557,HA20111156}, damage evolution \cite{MR3283795,MR3344593}, and crack propagation \cite{Yang2005,MR3831320,MR4659794}.   Unlike classical local models, NLMs employ  integral formulations that naturally capture nonlocal effects and relax the stringent regularity requirements on field variables, making them particularly suitable for problems involving discontinuities and singularities. 
	   
	   Building on these advantages, this work focuses on one-dimensional nonlocal elliptic interface (NLI) problems, where solutions  exhibit discontinuities across interfaces.  Specifically,  we consider the domain $\Omega=(a,b)$ with an interface point $\Gamma_0 = \alpha$ that divides $\Omega$ into two subdomains $\Omega_1=(a, \alpha)$ and $\Omega_2= (\alpha, b)$. 
To characterize nonlocal interactions, we introduce two horizon parameters $\delta_1$ and $\delta_2$ with $\delta_1 \leq\delta_2$, and define the interaction intervals as follows
\begin{itemize}
\item  $\mathcal{I}_1^D = (a-\delta_1,a)$  and $\mathcal{I}_2^D = (b,b + \delta_2) $ for Dirichlet-type nonlocal interactions;
\item $\mathcal{I}_1^J = (\alpha,\alpha + \delta_1)$ and  $ \mathcal{I}_2^J = (\alpha - \delta_2,\alpha)$  for interface interactions.
\end{itemize}
The nonlocal interface  $\Gamma$ is defined as the union of the interaction zones: $\Gamma = \mathcal{I}_1^J\cup \mathcal{I}_2^J =  (\alpha - \delta_2,\alpha + \delta_1)$. The complete interaction domains are denoted by 
$\mathcal{I}_1= \mathcal{I}_1^D\cup\mathcal{I}_1^J$ for $\Omega_1$ and  $\mathcal{I}_2 =\mathcal{I}_2^J\cup \mathcal{I}_2^D$ for $\Omega_2$, with the nonlocal Dirichlet boundary given by $\mathcal{I}^D =  \mathcal{I}_1^D\cup \mathcal{I}_2^D.$
For scalar functions $u_i: \Omega_i\cup\mathcal{I}_i\rightarrow \mathbb{R} \;(i = 1,2)$, the nonlocal operators are defined as
	\begin{align*}
		\mathcal{L}_{i}u_i(x)&=\int_{ \Omega_i\cup\mathcal{I}_i}[u_i(x)-u_i(y)]\gamma_{\delta_i}(x, y)\mathrm{d}y, \quad  x\in\Omega_i,
	\end{align*}
	where  the kernel function \(\gamma_{\delta_i}(x, y)\) satisfies
 \begin{equation}
 \label{kernel_properties}
\gamma_{\delta_i}(x, y): \mathbb{R} \times \mathbb{R} \to \mathbb{R}^+,\quad\;
\gamma_{\delta_i}(x, y) = \gamma_{\delta_i}(y, x),\quad\;
\gamma_{\delta_i}(x, y) = 0, \, |x - y| > \delta_i > 0.
\end{equation}
We consider a rescaled kernel  $\gamma_{\delta_i}(s) = \delta_i^{-3} \gamma_i ( s/\delta_i ),$ where  $\gamma_i$ is a nonnegative, radially symmetric function that is locally integrable and compactly supported within $|s|< 1$, 
and satisfies the moment condition
	\begin{equation}
		\label{kernel_scaling}
		 \dfrac{1}{2}\int_{-1}^1 s^2 \gamma_i(s) \, \mathrm{d}s=  \sigma_i >0,  \;\quad i = 1, 2.
	\end{equation}
The nonlocal interface (NLI) problem is to find  solutions $u_i (\,i=1,2)$ satisfying 
	\begin{equation}
		\label{non-interface-eq}
		\begin{split}
			\mathcal{L}_iu_i &= f_i, \quad x\in\Omega_i,\qquad {\rm and }\qquad
			u_i = g_i, \quad x\in \mathcal{I}_i^D,
		\end{split}
	\end{equation}
	with interface conditions
	\begin{equation}
		\label{non-interface-jump}
		\begin{split}
			u_2-u_1 &= \varphi, \quad x\in\Gamma,\qquad {\rm and }\qquad
			\mathcal{F}(u_1,u_2) = \psi, \quad x\in \Gamma\cup\Omega_2^J.
		\end{split}
	\end{equation}	
Here, $\Omega_2^J = (\alpha+\delta_1, \alpha+\delta_2)$ represents an extended interaction zone, and the interface-flux operator $\mathcal{F}(u_1, u_2)$ is defined  piecewise as in  \cite{MR4602510} 
	\begin{equation}
		\label{flux}
		\mathcal{F}(u_1, u_2) = 
		\begin{cases} 
			\displaystyle
			\int_{\Omega_2^J} [u_2(x) - u_2(y)]\gamma_{\delta_2}(x,y)\mathrm{d}y\\[4pt]
			\displaystyle		\qquad+ \dfrac{1}{2}\int_{\mathcal{I}_1^J} [u_1(x) - u_1(y)](\gamma_{\delta_2}(x, y)-\gamma_{\delta_1}(x, y))\mathrm{d}y, & x\in\mathcal{I}_2^J, \\[6pt]
			\displaystyle
			\dfrac{1}{2}\int_{\mathcal{I}_2^J} [u_2(x) - u_2(y)](\gamma_{\delta_1}(x, y)-\gamma_{\delta_2}(x, y))\mathrm{d}y, & x\in\mathcal{I}_1^J. 
		\end{cases}
	\end{equation}
A central challenge in NLI problems is twofold: ensuring the nonlocal solution converges to the classical local solution as nonlocal effects vanish (i.e., $\delta_1, \delta_2\to 0$), and developing efficient numerical methods capable of handling interfacial discontinuities without compromising accuracy or computational efficiency.

This challenge  raises a fundamentall question in nonlocal modeling: whether nonlocal interface equations indeed converge to a classical local interface problem as the nonlocal effect vanishes.  This limiting behavior, known as the {\em local limit}, is essential for practical applications of nonlocal models, especially in multiscale modeling and simulations, such as when macroscopic structures are influenced by  microscopic nonlocal effects. Various analytical techniques have been  employed to study this convergence,  including Talyor expansions \cite{MR1727557,SILLING201073}, functional analysis \cite{MR3164542, MR3492730, MR4818331}, and maximum principles \cite{du2019uniform,MR4374289}. Nevertheless, existing studies often lack proofs of local limit convergence in the maximum norm under realistic regularity assumptions, leaving a gap in solution stability analysis.

 Another crucial aspect is the asymptotic compatibility (AC) of numerical discretizations, that is whether discrete approximations also converges to the corresponding local limit as both the nonlocal horizon and mesh size tend to zero.  A significant advance was made by Tian and Du, who introduced the AC framework \cite{MR3231986,MR3143839}. This concept has since been generalized to Neumann-type constraints \cite{MR3621707},  spectral methods on periodic domains \cite{MR3514714,MR3591174},  discontinuous Galerkin schemes \cite{MR3854053,MR4812235}, nonlocal space–time models \cite{MR3639119}, and advanced gradient-recovery techniques \cite{MR3807955}. Recent contributions \cite{du2024errorestimatesfiniteelement,MR4464473} improved convergence analysis for finite element approximations with approximated interaction neighborhoods. However, many of these extensions either depend on specialized mesh structures or do not explicitly incorporate nonlocal jump conditions, thereby restricting their utility for NLI problems.
 
 Although classical interface problems have been widely studied using techniques such as  the immersed boundary method \cite{peskin1977numerical,peskin2002immersed}, extended/cut finite element method \cite{fries2010extended,hansbo2014cut} and immersed interface method \cite{he2011immersed,lin2013immersed,lin2015partially}, numerical approaches for NLI problems remain less developed.   
 Early formulations for NLI problems were established in  \cite{MR3367699, MR3125434}, yet they lacked rigorous error analysis. subsequent works \cite{MR4108294} established well-posedness and convergence under conforming discretizations, which can be computationally expensive for complex interfaces. State-based peridynamic models \cite{MR4208944, MR4455912} have addressed certain discontinuity issues but generally fail to achieve AC. More recently,  \cite{MR4602510} explored local limits and asymptotically compatible formulations for nonlocal interface problems with jumps. Despite these efforts, several key questions remain open: 
However, several critical issues remain unresolved:  (1) Can the local limit with respect to $\delta$ be rigorously proven  in the maximum norm under practical regularity assumptions?  (2) What is the optimal convergence rate with respect to $\delta$ for the continuum nonlocal models? (3) Is it possible to construct a theoretical framework that establishes the optimal convergence order for an AC unfitted finite element scheme applied to NLI problems?

Motivated by these open questions outlined above and the practical need to model engineering problems involving inherent discontinuities, such as fracture in composite materials or dynamic damage evolution, this work develops a robust  AC unfitted finite element method (FEM) on a Cartesian grid for one-dimensional nonlocal elliptic interface problems. In contrast to previous approaches \cite{MR4108294, MR4602510}, which often rely on conforming discretizations, the proposed method directly embeds the nonlocal jump conditions into the weak formulation via a Nitsche-type approach. This allows the use of a single Cartesian background mesh, greatly simplifying
implementation and improving computational efficiency.
The main contributions of this work  are summarized as follows:

 \textbf{(1) Local limit analysis in the maximum norm.} 
An extended maximum principle is established for NLI problems, providing explicit  control over solution behavior through the flux operator $\mathcal{F}(u_1,u_2)$.  Combined  with a rigorous asymptotic consistency analysis showing $\mathcal{F}(u_1, u_2) = \sigma_1 u_1'(\alpha) - \sigma_2 u_2'(\alpha) + \mathcal{O}(\max(\delta_1^2, \delta_2^2)),$ this work proves second-order convergence of nonlocal solutions to the local limit, i.e., $\|u - u_0\|_{\infty, \Omega_1 \cup \Omega_2} \leq C \max(\delta_1^2, \delta_2^2).$ This result fills a notable gap in the maximum-norm convergence analysis for nonlocal interface models.

\textbf{(2) Asymptotically compatible unfitted FEM scheme.} A Nitsche-type variational formulation is proposed to directly embed nonlocal jump conditions into the weak form,  thereby avoiding the need for body-fitted meshes. Computational efficiency is further improved via a novel block matrix decomposition strategy and an exact integration technique using linear transformations to reference elements. The resulting method preserves asymptotic compatibility while achieving optimal convergence rates, even for problems with discontinuous solutions across interfaces.

\textbf{(3) Optimal $(k+1)$-order convergence  analysis.}  The convergence of the proposed scheme is rigorously established, achieving order $(k+1)$ in both the energy and $L^2$-norms
\begin{align*}
		\|u - u^h\|_\delta \leq C \max\{\delta_1^{-1}, \delta_2^{-1}\} h^{k+1}, \quad 
		\|u - u^h\|_{L^2} \leq C \max\{\delta_1^{-1}, \delta_2^{-1}\} h^{k+1}.
	\end{align*}
The analysis accommodates arbitrary polynomial degree $k$,  allows non-matching horizon parameters $\delta_1\neq\delta_2$, and accurately captures solution jumps at multiple interfaces.  Numerical experiments  validate the robustness and high accuracy of the proposed scheme.

The paper is organized as follows.
Section 2 develops the maximum principle for nonlocal interface problems and proves second-order convergence to local limit under suitable constrained volume conditions.
Section 3 presents the asymptotically compatible unfitted finite-element method and derives corresponding optimal error estimates.
Section 4 provides numerical validation.
Conclusions are given in Section 5, and detailed proofs are collected in the Appendix.

	\section{The convergence limits from the nonlocal solutions to local solutions}
	\label{sec;nonlocal-local}
	\subsection{Maximum principle for nonlocal interface problem}
Maximum principle plays a central role in our analysis. Although well-established for local elliptic interface problems, its nonlocal counterpart must account for long-range interactions and interface jump conditions. We therefore establish a maximum principle for nonlocal elliptic interface problems to control the pointwise behavior of solutions.

To quantify the solution regularity, we introduce the space of bounded continuous functions  
	$
		C_b(\Omega) := \{u\in C(\Omega): u \text{ is bounded}\}, 
	$
	equipped with the supremum norm $\|u\|_{\infty, \Omega}:=\sup\limits_{x\in\Omega}|u(x)|.$
	Its natural extension  to higher regularity is defined as
		$
		C_b^m(\Omega) := \{u^{(k)} \in C_b(\Omega): \forall \text{ nonnegative integer} \; k \leq m \} 		$
		for any  integer \( m\geq 0 \), endowed with the  norm $
	\|u_i\|_{C_b^m(\Omega_i)} := \max\limits_{0 \leq k \leq m} \|u_i^{(k)}\|_{\infty,\Omega_i}. $ 
	For the interface problem, we further introduce the trace norm
	$\|u\|_{\infty, b} :=\sup_{x\in b}|u(x)|,$  where  $b= \Gamma,$ $\mathcal{I}_1^D$  or $\mathcal{I}_2^D$,
to capture the behavior of the solution on the interface and Dirichlet boundaries.
The global solution is measured by
$
\|u\|_{\infty, \Omega} = \max\limits_{i=1,2}\|u_i\|_{\infty, \Omega_i}$, providing a uniform control across subdomains. 
Endowed with these norms, the product space $C_b(\Omega_1)\times C_b(\Omega_2)$ forms a Banach space, which serves as the functional framework for establishing the nonlocal maximum principle and related interface properties. 
	\begin{theorem}
	\label{th1}
		Let $\Omega=(a,b)$ be partitioned into  \(\Omega_1 = (a,\alpha)\) and \(\Omega_2 = (\alpha,b)\),    and let \(\delta_1, \delta_2 > 0\)  be horizon parameters. 
		Suppose  \( u_i \in C(\Omega_i\cup\mathcal{I}_i) \)  for $i=1,2$  satisfying (i) \( \mathcal{L}_i u_i \leq 0 \) in \( \Omega_i \);  
		(ii)  $ u_1(x) = u_2(x) $,   $\forall  x \in  \Gamma $; and  (iii) $ \mathcal{F}(u_1,u_2)(x) \leq 0, \forall  x \in \Gamma $.
		Then the global maximum of $u$ is bounded by
		\[
		\|u\|_{\infty, \Omega_1 \cup \Omega_2} \leq \max\left( \|u_1\|_{\infty, \mathcal{I}_1^D}, \|u_2\|_{\infty, \mathcal{I}_2^D}\right).
		\]
	\end{theorem}
	\begin{proof}
Suppose $u$ attains a positive maximum $M$ at some point $x_0 \in \Omega_1$  and assume  this maximum is not achieved on $\mathcal{I}_1^D$. Since $u_1(x_0) = M \geq u_1(y)$ for all $y\in \Omega_1 \cup \mathcal{I}_1$, the nonlocal operator $\mathcal{L}_1 u_1(x)$ evaluated  at $x_0$ satisfies
\[
    \mathcal{L}_1 u_1(x_0) = \int_{\Omega_1 \cup \mathcal{I}_1} [u_1(x_0) - u_1(y)] \gamma_{\delta_1}(x_0, y) \, dy \geq 0. 
\]
On the other hand, condition {\rm (i)} implies $\mathcal{L}_1 u_1(x_0) \leq 0$, so it follows that $\mathcal{L}_1 u_1(x_0) = 0$. This equality forces$u_1(y) = M$ for all $y \in B_{\delta_1}(x_0) \cap (\Omega_1 \cup \mathcal{I}_1)$. By continuity, $u_1 \equiv M$ in a neighborhood of $x_0$. From condition (ii), it follows that $u_2(\alpha) = u_1(\alpha) = M$. 

Now consider $x \in \mathcal{I}_1^J $.  A direct calculation shows that the flux operator satisfies
\[
\mathcal{F}(u_1, u_2)(x) = \frac{1}{2} \int_{\mathcal{I}_2^J} [u_2(x)- u_2(y)] (\gamma_{\delta_2} - \gamma_{\delta_1}) \, dy = \frac{1}{2} \int_{\mathcal{I}_2^J} [u_2(x)- M] (\gamma_{\delta_2} - \gamma_{\delta_1}) \, dy \leq 0.
\]
The sign-indefinite nature of $\gamma_{\delta_2} - \gamma_{\delta_1}$ forces $u_2(x) = M$ throughout $\mathcal{I}_1^J$.
Similarly,  for any $x \in \mathcal{I}_2^J$, the flux condition $\mathcal{F}(u_1, u_2)(x)\leq 0$ implies 
 $
    \int_{\Omega_2^J} [u_2(x) - u_2(y)] \gamma_{\delta_2}(x,y) \, dy \leq 0.
$
Since $u_2(x) = M \geq u_2(y)$, equality holds if and only if  $u_2(y) = M$ for all $y \in \Omega_2^J$.
     
For $x \in \Omega_2$ with $\text{dist}(x, \alpha) \leq \delta_2$, the nonlocal operator $\mathcal{L}_2 u_2(x)$ becomes
    \[
    \mathcal{L}_2 u_2(x) = \int_{\Omega_2 \cup \mathcal{I}_2} [u_2(x) - u_2(y)] \gamma_{\delta_2}(x, y) \, dy = \int_{\Omega_2 \cup \mathcal{I}_2} [M - u_2(y)] \gamma_{\delta_2}(x, y) \, dy \leq 0.
    \]
    Any value $u_2(y) > M$ would lead to $\mathcal{L}_2 u_2(x) > 0$,  contradicting the inequality. Hence $u_2(y) = M$ for all $y\in \Omega_2 \cup \mathcal{I}_2$. Finally, if  $u_2(x) < M$ for some $x \in \mathcal{I}_2^D$,  then
    \[
    \mathcal{L}_2 u_2(x) = \int_{b - \delta_2}^{b + \delta_2} [u_2(x) - u_2(y)] \gamma_{\delta_2}(x, y) \, dy = \int_{b - \delta_2}^{b+\delta_2} [u_2(x) - M] \gamma_{\delta_2}(x,y) \, dy < 0,
    \]
  which  contradicts the condition (i).
    Therefore, the maximum $M$ must occur on $\mathcal{I}_2^D$.
\end{proof}

To clearly distinguish the present results from both nonlocal formulations and their local counterparts, pointwise a priori estimates are now presented.	
	\begin{theorem}\label{Priori-non}
		Under the same assumptions as Theorem \ref{th1}, with additional regularity condition that  $u_i \in C_b^2(\Omega_i \cup \mathcal{I}_i^D)$ for $i=1,\,2$, there exists a constant $C = C(\gamma_1, \gamma_2, |\Omega|) $,  independent of the horizon parameters $\delta_1$ and $\delta_2$ such that  the following estimate holds
		\[
		\|u\|_{\infty, \Omega} \leq C \big(\|u\|_{\infty, \mathcal{I}^D} + \|\mathcal{L} u\|_{\infty, \Omega} + \|\mathcal{F}(u_1, u_2)\|_{\infty, \Gamma} \big).
		\]
	\end{theorem}
	
	\begin{proof}
	An auxiliary function  $v(x)$ is constructed piecewise on $\Omega \cup \mathcal{I}^D$ as follows
		\[v(x) = \begin{cases}
			v_1(x) =
			A_1 + B_1|x-\alpha|^2 + \epsilon \phi_1(x),  \;\;x \in \Omega_1 \cup \mathcal{I}_1^D,\\
			v_2(x) =
			A_2 + B_2|x-\alpha|^2 + \epsilon \phi_2(x),  \;\;x \in \Omega_2 \cup \mathcal{I}_2^D,\\
		\end{cases}
		\]
		where $A_i = \|u_i\|_{\infty, \mathcal{I}_i^D} + |\Omega|^2/2\sigma_i$ controls the boundary values, $B_i =-1/2\sigma_i < 0$,  $\phi_i\in C_b^2(\Omega_i \cup \mathcal{I}_i^D)$ satisfies $\mathcal{L}_i \phi_i = 1$ with $\phi_i|_{\mathcal{I}_i^D} = 0$, and  $\epsilon > 0$ is a parameter to be determined.  By the nonlocal maximum principle \cite{du2019uniform} and the condition $\mathcal{L}_i \phi_i = 1$ with $\phi_i|_{\mathcal{I}_i^D} = 0$, there exists a constant $C_\phi = C_\phi(\gamma_i)$ such that
		$
		\|\phi_i\|_{\infty, \Omega_i} \leq C_\phi.
		$
		
		Now define  $w_i = u_i - v_i$. On the Dirichlet boundary $\mathcal{I}_i^D$,
		\begin{align*}
v_i(x)\big|_{\mathcal{I}_i^D} &= A_i + B_i|x-\alpha|^2 \geq \left(\|u_i\|_{\infty, \mathcal{I}_i^D} + \frac{|\Omega|^2}{2\sigma_i}\right) - \frac{|\Omega|^2}{2\sigma_i}  = \|u_i\|_{\infty, \mathcal{I}_i^D} \geq u_i(x)\big|_{\mathcal{I}_i^D},
\end{align*}
		which implies \( w_i  \leq 0 \)  on \( \mathcal{I}_i^D\). Within \( \Omega_i \), 
		noting that  \( \mathcal{L}_i A_i = 0 \),  \( \mathcal{L}_i \phi_i = 1 \), and 
		\begin{align*}
B_i \mathcal{L}_i(|x - \alpha|^2) = B_i\int_{\Omega_i \cup \mathcal{I}_i} \left[|x - \alpha|^2 - |y - \alpha|^2\right] \gamma_{\delta_i}(x,y) \, dy 
= -2B_i \sigma_i = 1,
\end{align*}
it  follows that  
		\begin{align*}
		\mathcal{L}_i w_i &= \mathcal{L}_i u_i - \Big( \mathcal{L}_i (A_i + B_i|x - \alpha|^2) + \epsilon \mathcal{L}_i \phi_i \Big)=\mathcal{L}_i u_i - (1 + \epsilon) \leq 0. 
		\end{align*}
		By the linearity of flux operator $\mathcal{F}$ and integrability of $\gamma_{\delta_i}$, the estimate 
		$|\mathcal{F}(w_1, w_2)| \leq C_F \epsilon $  hold,  where  $C_F$  depends on $\|\phi_i\|_\infty $  and  $ \gamma_{\delta_i} $.
		Choosing  $\epsilon =(\|\mathcal{F}(u_1,u_2)\|_{\infty, \Gamma} + \|\mathcal{L}_i u_i\|_{\infty,\Omega_i})/C_F$, yields 
		\[
		\mathcal{F}(w_1,w_2) \leq \mathcal{F}(u_1,u_2) + \mathcal{L}_i u_i \leq 0 \quad \text{on } \Gamma.
		\]
		Applying the nonlocal interface maximum principle (Theorem \ref{th1}), it is concluded that $w_i\leq 0$ in $\Omega_i$,   implying  \( u_i \leq v_i \).
		Therefore,
		\[
		\begin{split}
			\|u_i\|_{\infty, \Omega_i} \leq \|v_i\|_{\infty, \Omega_i} &= \max_{x \in \Omega_i} \left| A_i + B_i|x - \alpha|^2 + \epsilon \phi_i(x) \right| \leq A_i + \frac{|\Omega|^2}{2\sigma_i} + \epsilon C_\phi.
		\end{split}
		\]
		Substituting the expressions for $\epsilon$ and $A_i$ gives
		\[
		\|u_i\|_{\infty, \Omega_i}  \leq  \|u_i\|_{\infty, \mathcal{I}_i^D} + \dfrac{|\Omega|^2}{\sigma_i} + \frac{C_\phi}{C_F} (\|\mathcal{F}(u_1, u_2)\|_{\infty, \Gamma}+ \|\mathcal{L}_i u_i\|_{\infty,\Omega_i}).
		\]
		The constant $C$ can be taken as
		$
		C = \max\left(1, \frac{|\Omega|^2}{\min(\sigma_1, \sigma_2)}, \frac{C_\phi}{C_F}\right),
		$
		which depends on $\gamma_i$ (through $\sigma_i$, $C_\phi$, $C_F$) and $|\Omega|$, but remains independent of $\delta_1$ and  $\delta_2$. 
	\end{proof}

	\subsection{Approximation order between nonlocal elliptic interface problem to its local counterpart}
	This subsection  is devoted to analyzing the  asymptotic convergence rate of the nonlocal solution $u_i$  towards its local counterpart $u_{0i}$ as the horizon parameter $\delta_i\rightarrow 0$.  
The corresponding local interface problem is given by 
	\begin{equation}
		\label{interface-eq}
			\mathcal{L}_i^0 u_{0i}(x) :=-\sigma_i\dfrac{d^2u_i(x)}{dx^2}=f_{i}(x), \; x\in\Omega_i,\;i =1,2, 
	\end{equation}
	subject to  boundary conditions $u_{01}(a)=g_{01},\ u_{02}(b) =g_{02}$, and interface conditions 
	\begin{equation}
		\label{interface-jump}
			[\![u_0]\!]=0,  \;\qquad
			[\![\sigma\frac{\partial u_0}{\partial n}]\!] = \psi(x), \; x\in \Gamma_0,
	\end{equation}
	where $f_i\in C_b^2(\Omega_i)$ and  $\sigma$ takes the value  $\sigma_1$ in $\Omega_1$ and $\sigma_2$  in $\Omega_2$.  
	
	For  a function $\hat{u}_i\in C_b^4(\Omega_i\cup \mathcal{I}_i)$, the nonlocal operator $\mathcal{L}_i$ approximates the local operator $\mathcal{L}_i^0$ with second-order accuracy  at any point $x\in\Omega_i$.  As established in prior studies  (e.g., \cite{du2019uniform,MR3938295}), this relationship can be expressed as
	\begin{equation}\label{eq_convergence}
		\mathcal{L}_{i}\hat{u}_i(x)=\mathcal{L}_i^0\hat{u}_i(x)+\mathcal{O}(\delta_i^2).
	\end{equation}
A natural question arises: does the interface flux operator  $\mathcal{F}(u_1, u_2)$ exhibit analogous second-order convergence? An affirmative answer is essential to ensure consistency between the nonlocal and local interface formulations. The following lemma provides such a result.	\begin{lemma} 
	\label{lemma:flux_limit}
		Let $\hat{u}_i \in C_b^4(\Omega_i\cup\mathcal{I}_i)$  for $i=1,2$. Assume the kernels $\{\gamma_{\delta_i}\}$ satisfies \eqref{kernel_properties}, \eqref{kernel_scaling}  and
	\begin{equation}
	\label{kernel_bound}
	G(\gamma_{\delta_i}):=\sup_{x\in\Omega_i\cup\mathcal{I}_i}\int_{B_{\delta_i}\cap(\Omega_i\cup\mathcal{I}_i)}\gamma_{\delta_i}(x,y)\mathrm{d}y\lesssim \delta_i^{-2}.
	\end{equation}
	 Then, for $x \in  \Gamma $, the nonlocal flux operator $\mathcal{F}(u_1, u_2)$ converges to its local counterpart as  
	  		\begin{equation}\label{eq:flux_convergence}
			\mathcal{F}(\hat{u}_1, \hat{u}_2) = \sigma_1 \hat{u}_1'(\alpha) - \sigma_2 \hat{u}_2'(\alpha) + \mathcal{O}(\max(\delta_1^2, \delta_2^2)).
		\end{equation}
	\end{lemma}
	
	For clarity of exposition, the detailed proof is deferred to Appendix \ref{app:flux_proof}.
	
	\begin{theorem} 
		\label{thm:nonlocal_to_local}
		Suppose $u_0 \in C_b^4(\Omega_1 \cup \Omega_2)$ is the unique solution to the local interface problem \eqref{interface-eq}-\eqref{interface-jump}, and assume the family of kernels $\{\gamma_{\delta_i}\}$ satisfies \eqref{kernel_properties}, \eqref{kernel_scaling} and \eqref{kernel_bound}.
		Let  $u$ be the unique solution to the nonlocal interface problem \eqref{non-interface-eq}-\eqref{non-interface-jump} with boundary conditions  constructed as
		\[
		g_i(x) = g_{0i} + (x - x_i)\dfrac{du_{0i}}{dx}(x_i), \quad x_i = a \text{ or } b, i=1,2.
		\]
		Then there exists a constant $C > 0$, independent of $\delta_1, \delta_2$, such that
		\[
		\|u - u_0\|_{\infty, \Omega_1 \cup \Omega_2} \leq C \max(\delta_1^2, \delta_2^2).
		\]
	\end{theorem}
	
	\begin{proof}
		Let $\widetilde{u}_0$ be the piecewise fourth-order Taylor approximation of $u_0$ defined by 
		\begin{equation*}
			\widetilde{u}_0(x) =
			\begin{cases}
				\displaystyle
				\sum\limits_{m=0}^4 \dfrac{(x - a)^m}{m!} \partial_x^m{u_{01}(a)}, & x \in \mathcal{I}_1^D, \nonumber \\
				\displaystyle u_{01}(x), & x \in \Omega_1, \nonumber \\
				\displaystyle u_{02}(x), & x \in \Omega_2, \nonumber \\
				\displaystyle \sum\limits_{m=0}^4 \dfrac{(x - b)^m}{m!} \partial_x^m u_{02}(b), & x \in \mathcal{I}_2^D. \nonumber
			\end{cases}
		\end{equation*}
		Since $u_0 \in C_b^4(\Omega_1\cup\Omega_2)$, the extension $\widetilde{u}_0$ belongs to  $C_b^4(\mathcal{I}_1^D\cup\Omega_1\cup\Omega_2\cup\mathcal{I}_2^D)$.
		
		For $x \in \Omega_i$, the second-order accuracy of the nonlocal operator \eqref{eq_convergence} implies
		\begin{equation}
		\label{eq:th3-1}
			\mathcal{L}_{i} \widetilde{u}_0 = \mathcal{L}_i^0 \widetilde{u}_0 + \mathcal{O}(\delta_i^2) = \mathcal{L}_i^0 u_{0i} + \mathcal{O}(\delta_i^2)
			= f_{i} + \mathcal{O}(\delta_i^2),\quad i=1,2.
		\end{equation}
		Define the error $e = u - \widetilde{u}_0$. From \eqref{non-interface-eq} and \eqref{eq:th3-1}, it follows that 
		\begin{equation}\label{thm:nonlocal_to_local-eq1}
			\mathcal{L}_i e  =\mathcal{L}_i (u - \widetilde{u}_0)  =  \mathcal{O}(\delta_i^2), \quad x \in \Omega_i.
		\end{equation}
		On the Dirichlet region $\mathcal{I}_i^D$,   the Taylor remainder estimate gives
		\begin{equation}\label{thm:nonlocal_to_local-eq2}
		\|e\|_{\infty, \mathcal{I}_i^D} \leq \frac{\delta_i^4}{24} \left\|\partial_x^4 u_{0i}\right\|_{\infty, \mathcal{I}_i^D}.
		\end{equation}
		By Lemma \ref{lemma:flux_limit},  and the construction of $\tilde{u}_0$,  the flux error satisfies
		\begin{equation}\label{thm:nonlocal_to_local-eq3}
		\mathcal{F}(e_1, e_2)= \mathcal{F}(u_1, u_2) -
		\mathcal{F}(\tilde{u}_{01}, \tilde{u}_{02})  =\psi - \Big(\sigma_1 \tilde{u}_{01}'(\alpha) - \sigma_2 \tilde{u}_{02}'(\alpha) +\mathcal{O}(\max(\delta_1^2, \delta_2^2) \Big).
		\end{equation}
		Since $\tilde{u}_{0i}'(\alpha) = u_{0i}'(\alpha) + \mathcal{O}(\delta_i^2)$  and  $\sigma_1 u_{01}'(\alpha) - \sigma_2 u_{02}'(\alpha) = \psi$, it is conclude that
		\[
		\|\mathcal{F}(e_1, e_2)\|_{\infty, \Gamma} = \mathcal{O}(\max(\delta_1^2, \delta_2^2)).
		\]
		Applying the a priori estimate from Theorem \ref{Priori-non} yields
		\[
		\|e\|_{\infty, \Omega_i} \leq \left( \max_i \|e\|_{\infty, \mathcal{I}^D_i} + \max_i \|\mathcal{L}_i e\|_{\infty, \Omega_i} + \|\mathcal{F}(e_1, e_2)\|_{\infty, \Gamma} \right).
		\]
Substituting the bounds in \eqref{thm:nonlocal_to_local-eq1} - \eqref{thm:nonlocal_to_local-eq3} into the above inequality gives
\[\|e\|_{\infty, \Omega_1\cup\Omega_2} \leq C\max(\delta_1^2, \delta_2^2).
\]
Finally, since  $\widetilde{u}_0=u_0$ on $\Omega_1\cup\Omega_2$, the triangle inequality implies
		\[
		\|u - u_0\|_{\infty, \Omega_1 \cup \Omega_2} \leq \|e\|_{\infty, \Omega_1 \cup \Omega_2} + \|\widetilde{u}_0 - u_0\|_{\infty, \Omega_1 \cup \Omega_2} \leq C \max(\delta_1^2, \delta_2^2),
		\]
	which completes the proof. \end{proof}
	
	\begin{remark}
		If the Dirichlet boundary conditions for the nonlocal problem are prescribed directly as
	$u_1 = g_{01}$ on $\mathcal{I}_1^D$, and  $u_2 = g_{02}$ on  $\mathcal{I}_2^D$,
		where $g_{01}$ are taken from the local problem, then the convergence order in  Theorem  \ref{thm:nonlocal_to_local}
		 reduces to  $\max(\delta_1, \delta_2)$.  This is due to the mismatch between the local boundary data and the nonlocal extension, which introduces a first-order error in the boundary region.
	\end{remark}
	
	\section{Unfitted finite element method for nonlocal interface problem}
	\label{sec;unfitted-FEM}
This section presents the finite element discretization of the nonlocal interface problem \eqref{non-interface-eq}-\eqref{non-interface-jump}.
In contrast to approaches that align element boundaries with the interface $\Gamma_0$ to capture flux jumps, an unfitted methodology is adopted. The scheme employs discontinuous piecewise polynomial finite  element spaces on a uniform background grid,  thereby eliminating any dependence of the discretization on the interface location.

\subsection{Weak formulation}
The solution and test functions are defined piecewise as follows 
	\[
	u = 
	\begin{cases}
		u_1 & \text{in } \Omega_1\cup\mathcal{I}_1, \\
		u_2 & \text{in } \Omega_2\cup\mathcal{I}_2,
	\end{cases}
	\qquad\qquad
	v = 
	\begin{cases}
		v_1 & \text{in } \Omega_1\cup\mathcal{I}_1, \\
		v_2 & \text{in } \Omega_2\cup\mathcal{I}_2.
	\end{cases}
	\]
	The \textit{nonlocal energy norm} $\|\cdot\|_{\delta_i}  := \sqrt{(u, u)_{\delta_i}}$ is induced by inner product
	\[
	(u, v)_{\delta_i} := \frac{1}{2} \int_{\Omega_i \cup \mathcal{I}_i^D} \int_{\Omega_i \cup \mathcal{I}_i^D} [u_i(x) - u_i(y)][v_i(x) - v_i(y)] \gamma_{\delta_i}(x, y) \, dy \, dx. 
	\]
To account for coupling effects across the interface $\Gamma$, an additional inner product is introduced
	\[
	\begin{split}
		(u, v)_{\Gamma}&:=\frac{1}{2}\sum_{i = 1}^{2}\int_{\mathcal{I}_j^J}\int_{\mathcal{I}_i^J}[u_i(x) - u_i(y)][v_i(x) - v_i(y)]\gamma_{\delta_i}(x, y)dydx\\
		&\;\quad+\int_{\mathcal{I}_2^J}\int_{\Omega_2^J}[u_2(x) - u_2(y)][v_2(x) - v_2(y)]\gamma_{\delta_2}(x, y)dydx,\\
	\end{split}
	\]
	with the corresponding norm $
	\|u\|_\Gamma := \sqrt{(u, u)_\Gamma}. $
The global nonlocal energy norm is  then defined as
	$
	\|u\|^2_\delta := {\|u\|_{\delta_1}^2 + \|u\|_{\delta_2}^2 + \|u\|_\Gamma^2}.
	$
The  energy space is given by 
$$
V_{\delta}:= \{v\in L^2(\Omega\cup\mathcal{I}^D):\; \|u\|_\delta <\infty,\;v_i =0 {\;\rm in\;} \mathcal{I}_i^D,\; v_1=v_2  {\rm\; in\;}  \Gamma\},
$$
and its dual space is denoted by $V'_{\delta}$. The weak  formulation of \eqref{non-interface-eq}-\eqref{non-interface-jump} reads: given $f\in V'_{\delta}$ and $\psi\in V'_{\delta}$, find $u\in V_{\delta}$ such that
	\begin{equation}
	\label{non-weak-form}
	a(u_1,u_2,v_1,v_2)=l(v_1,v_2), \quad {\rm for\; all\; } v\in V_{\delta},
	\end{equation}
	where the bilinear and linear forms are respectively defined as
	\begin{align*}
		a(u_1, u_2, v_1, v_2) &= \frac{1}{2} \sum_{i = 1}^{2} \int_{\Omega_i\cup\mathcal{I}_i^D} \int_{\Omega_i\cup\mathcal{I}_i^D} (u_i(x) - u_i(y))(v_i(x) - v_i(y)) \gamma_{\delta_i}(x, y) dydx \\
		& \quad + \frac{1}{2} \sum_{i = 1}^{2} \int_{\mathcal{I}_j^J} \int_{\mathcal{I}_i^J} (u_i(x) - u_i(y))(v_i(x) - v_i(y)) \gamma_{\delta_i}(x, y) dydx \\
		& \quad + \int_{\mathcal{I}_2^J} \int_{\Omega_2^J} (u_2(x) - u_2(y))(v_2(x) - v_2(y)) \gamma_{\delta_2}(x,y) dydx,\\
		l(v_1, v_2) &= \sum_{i = 1}^{2} \int_{\Omega_i} f_i(x) v_i(x) \, dx + \int_{\mathcal{I}_2^J} \psi_1(x) v_1(x) \, dx + \int_{\mathcal{I}_1^J} \psi_2(x) v_2(x) \, dx.
	\end{align*}
The well-posedness of the weak formulation  \eqref{non-weak-form} has been established in \cite{MR4602510}.

	\subsection{Unfitted finite element discretization}	
	A uniform partition $\mathcal{T}^h$ of the computational domain $\Omega \cup \mathcal{I}^D$  is constructed,  with nodes distributed as
	$
	a-\delta_1 = x_0  < \cdots < x_{M_1-1} < a = x_{M_1}  < \cdots < x_{M_1+N} = b <  \cdots < x_{N+M_1+M_2} = b+\delta_2.
	$
Notably, the interface point  $\alpha\in(a, b)$ is not required to align with any grid nodes.  
The mesh parameter $h$ denotes the length  of uniform subinterval $K = (x_j, x_{j+1})$, for $ j = 0, 1, \cdots, N+M_1+M_2-1$. To handle the nonlocal interface $\Gamma=(\alpha-\delta_2, \alpha+\delta_1)$,  elements are classified as follows 
\begin{itemize}
\item Interface-intersecting elements: $\mathcal{T}_{\Gamma}^h=\{T\in\mathcal{T}^h: T\cap\Gamma\neq\varnothing\},$
\item Nonlocal interaction elements:
$
\mathcal{T}_{i, J}^h=\{T\!\in\!\mathcal{T}^h\!:\! T\cap\mathcal{I}_i^J\neq\varnothing\}, i=1,2, \mathcal{T}_{\Omega_2, J}^h=\{T\!\in\!\mathcal{T}^h\!:\! T\cap\Omega_2^J\neq\varnothing\},
$
\item Subdomain elements:
	$\mathcal{T}_i^h=\{T\in\mathcal{T}^h: \overline{T}\cap\overline{(\Omega_i\cup\mathcal{I}_i^D)}\neq\varnothing\},\quad i = 1,2.$
\end{itemize}
The corresponding  finite element spaces are defined as
$$
	V_i^h=\left\{v\in L^2(\Omega_i\cup\mathcal{I}_i^D):\left.v\right|_K\in\mathbb{P}^k(K),\ \forall K\in\mathcal{T}_i^h\right\},\;\;
	V^h=V_1^h \times V_2^h,
$$
with the constrained space 
$
V^{0,h}=\left\{v^h\in V^h:\left.v^h\right|_{\mathcal{I}^D}=0\right\},
$
where $\mathbb{P}^k(K)$ denotes by the space of polynomials of degree at most $k$ on element $K$. 
The discrete weak formulation  seeks $ u^h=(u_1^h,u_2^h)\in V^h$ such that  
	\begin{equation}
	\label{DG-formulation}
		a(u_1^h,u_2^h, v_1^h,v_2^h)=l(v_1^h,v_2^h),\quad \forall v^h\in V^{0,h}.
	\end{equation}
%
%
	
\subsection{Implementation}
	\begin{figure}[H]
		\centering
		\subfigure{\includegraphics[scale=0.17]{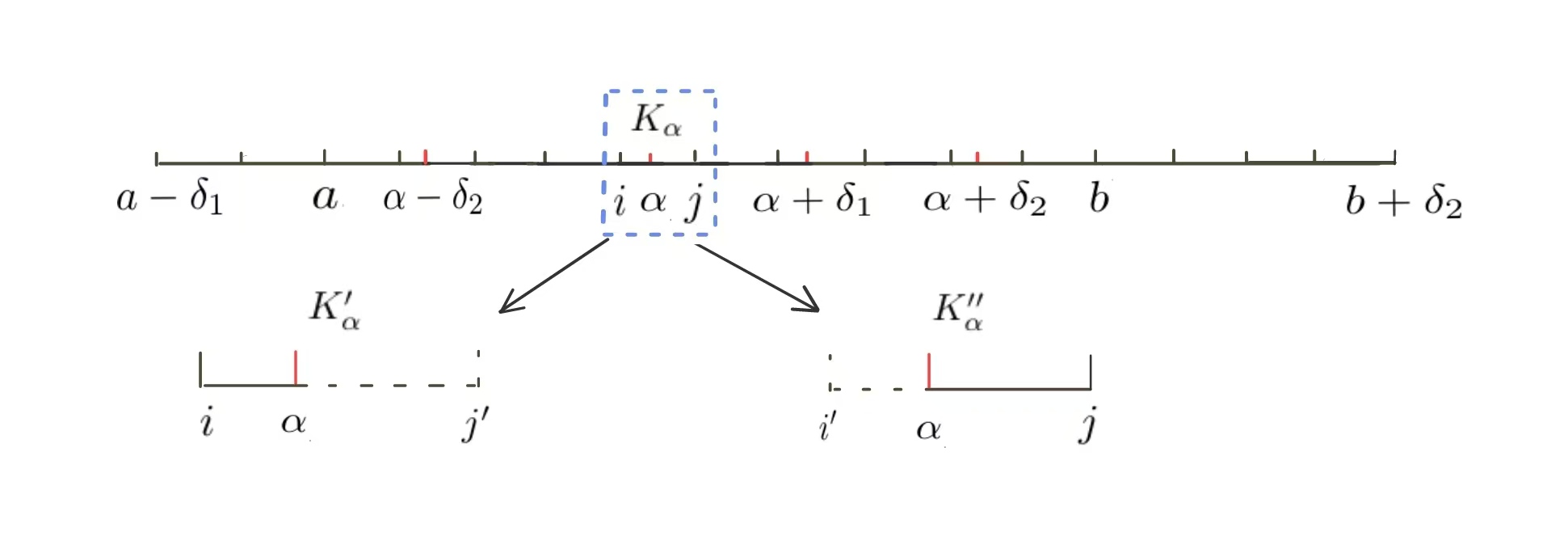}}
		     \vspace{-0.8cm}
		\caption{Schematic of the non-interface-conforming uniform mesh}
		\label{uniform-mesh}
	\end{figure}

For illustration, the lowest-order case is considered, using piecewise-linear basis functions.
The algorithm first identifies the element $K_{\alpha}$ intersected by the local interface $\Gamma_0$. This element is then subdivided into a left sub-interval $K'_{\alpha}$ and a right sub-interval $K''_{\alpha}$.  On $K'_{\alpha}$, a virtual nodes $j'$ is introduced on the far side of $\Gamma_0$, while on $K''_{\alpha}$, a virtual nodes $i'$ is placed on the opposite  side (see Figure~\ref{uniform-mesh}). 
These nodes serve as algebraic placeholders:  they coincide geometrically with $\Gamma_0$, but are not conventional finite element vertices.
Their purpose is to decouple the discrete function spaces on $\Omega_1$ and $\Omega_2$ while preserving the physical geometry.
A consistency check guarantees that no spurious duplication occurs at element boundaries.
Upon completion, two independent meshes are obtained: one for $\Omega_1$ and another for $\Omega_2$, each containing interface-adjacent elements that are geometrically coincident but topologically distinct.

For the finite element discretization, the solution $u^h$ and test functions $v^h$ are  expanded in the piece-wise linear basis functions $\{\phi_m\}_{m=1}^{M_d}$, where the index set $M_h$ includes all physical and virtual degrees of freedom.
Substitution into the weak formulation \eqref{DG-formulation} yields the linear system
\[
\mathbb{A} \mathbf{U}_h = \mathbf{b},
\]
where $\mathbf{U}_h$ is the vector of nodal unknowns, $\mathbb{A}$ is the global stiffness matrix, and $\mathbf{b}$ is  the load vector.
To streamline assembly,  the matrix is decomposed additively as
\[
\mathbb{A} = \mathbb{A}_{\Omega} + \mathbb{A}_{\Gamma}. 
\]
Here,  $\mathbb{A}_{\Omega}$ collects standard nonlocal element contributions over $\mathcal{T}^h_i$, while $\mathbb{A}_{\Gamma}$ encodes the interface terms integrated over $\mathcal{T}_{\Gamma}^h \cup \mathcal{T}_{\Omega_2, J}^h$. In one dimension, the computation of  $\mathbb{A}_{\Gamma}$ is exact and efficient: a simple affine mapping to the reference element reduces all interface integrals to closed-form expressions. Although the virtual nodes are essential for embedding the jump conditions directly into the weak form, they are purely algebraic artifacts and do not belong to the physical domain $\Omega$.
This approach captures discontinuities across the interface without body-fitted meshing, underscoring the method's robustness for nonlocal interface problems.

	\subsection{Error estimate for the unfitted FEM}	
	A priori error estimates are now established for the unfitted finite element discretization of nonlocal interface problems, quantifying  approximation errors in both energy and $L^2$ norms.
The analysis relies on two key lemmas, whose proofs are provided in Appendix \ref{app:norm_proof}.
\begin{lemma}
	\label{lem:energy_bound}
		Suppose the kernel family $\{\gamma_{\delta_i}\}$ satisfies \eqref{kernel_properties},  \eqref{kernel_scaling} and \eqref{kernel_bound}. Then there exists a  constant $C>0$ such that for all piecewise continuous bounded functions $v=(v_1,v_2)$ with $v_i \in C_b(\Omega_i \cup \mathcal{I}_i^D)$, the following bound holds
		\[
		\|v\|_\delta \leq C \max\{\delta_1^{-1}, \delta_2^{-1}\} \left(\|v\|_{\infty, \Omega_1\cup\Omega_2} + \|v\|_{\infty, \mathcal{I}_1^D\cup\mathcal{I}_2^D}\right).
		\]
	\end{lemma}
	\begin{lemma}
	\label{lem:poincare}
		Under the same kernel assumptions,  there exist constants $C_1, C_2 > 0$, depending only on $\Omega$ and $\gamma_i$, such that for all $v \in V^0 = \{v \in L^2(\Omega_1 \cup \Omega_2) : v|_{\mathcal{I}_i^D} = 0\}$, 
		\begin{equation}\label{L^2}
			C_1\|v\|_{L^2(\Omega_1\cup\Omega_2)} \leq \|v\|_\delta \leq C_2\max\{\delta_1^{-1},\delta_2^{-1}\}\|v\|_{L^2(\Omega_1\cup\Omega_2)}.
		\end{equation}
	\end{lemma}

The main convergence result for the unfitted FEM is now presented.

	\begin{theorem}
	\label{thm:FEM_error} 
	Under the  assumptions of Theorem \ref{thm:nonlocal_to_local}, 
	if $\tilde{u}_0$ is a $C^{\max(4,k+1)}_b(\Omega_1 \cup \Omega_2)$ extension of $u_0$, and $u^h$ is the finite element solution in $V^h$ with $k$-th degree Lagrange  interpolation of boundary data $g_i^h$, then there exists a constant $C>0$, independent of $\delta_1,\delta_2$ and $h$, such that
	\begin{eqnarray}
		\|u - u^h\|_\delta \leq C \max\{\delta_1^{-1}, \delta_2^{-1}\} h^{k+1}, \label{error-delta}\\
		\|u - u^h\|_{L^2(\Omega_1 \cup \Omega_2)} \leq C \max\{\delta_1^{-1}, \delta_2^{-1}\} h^{k+1}.\label{error-l2}
	\end{eqnarray}
	\end{theorem}
	\begin{proof}
	Let \(u^*=(u_1^*, u_2^*) \in V^h\)  be a projection of the exact nonlocal solution $u$ onto the finite element space, satisfying
	\begin{equation}\label{m-eq}
		\mathcal{L}_i u_i^* = f_i, \;\text{ in }\; \Omega_i, \quad u_i^* = g_i^h,\; \text{ on }\; \mathcal{I}_i^D,
	\end{equation}
	and the jump conditions 
	\begin{equation}\label{m-jump}
		\llbracket u^* \rrbracket = 0,  \quad \mathcal{F}(u_1^*, u_2^*)=\psi(x), \;\text{ on } \Gamma.
	\end{equation}
	From \eqref{non-interface-eq} and \eqref{m-eq},  the residual satisfies
	\[\mathcal{L}_{i}(u_i-u_{i}^{*})=0, \;x\in\Omega_{i}.\]
	The interface conditions \eqref{non-interface-jump} and \eqref{m-jump} imply
	\[\mathcal{F}(u_i-u_{i}^{*}, u_i-u_{i}^{*})=0, \;{\rm in} \;\Gamma.\]
	Applying the nonlocal maximum principle (Theorem \ref{th1}) yields
	\begin{equation}
	\label{I-error}
		\begin{split}
			\|u - u^{*}\|_{\infty,\Omega_1\cup\Omega_2}&\leq\|u-u^{*}\|_{\infty,\mathcal{I}_1^D\cup\mathcal{I}_2^D}+C\|\mathcal{L}(u-u^*)\|_{\infty,\Omega_1\cup\Omega_2} + \|\mathcal{F}(u_{1}-u_{1}^{*}, u_{2}-u_{2}^{*})\|_{\infty,\Gamma}\\
			&\leq C_{I}h^{k+1}\|u\|_{C^{k+1}_b(\mathcal{I}_1^D\cup\mathcal{I}_2^D)}.
		\end{split}
	\end{equation}
	Lemma \ref{lem:energy_bound} then gives 
	\begin{align*}
		\|u - u^*\|_{\delta} & \leq C \max\{\delta_1^{-1}, \delta_2^{-1}\}h^{k+1}\|u\|_{C^{k+1}_b(\mathcal{I}_1^D\cup\mathcal{I}_2^D)}.
	\end{align*}
	Now consider the discrete error $u^* - u^h$.   By  Galerkin orthogonality of \eqref{DG-formulation}, 
	 \begin{equation}
	 \label{orthogonality}
	 a(u^*-u^h, v^h)=0,\quad\forall \;v^h\in V^{0,h}.
	 \end{equation}
	  Let \(I_h \tilde{u}_0\) be the piecewise $k$-degree polynomial interpolation of  $\tilde{u}_0$. Standard interpolation theory gives 
	 \[\|\tilde{u}_0 - I_h\tilde{u}_0\|_{\delta} \leq C\max\{\delta_1^{-1}, \delta_2^{-1}\}\|\tilde{u}_0 - I_h\tilde{u}_0\|_{\infty, \Omega_1\cup\Omega_2} \leq C\max\{\delta_1^{-1}, \delta_2^{-1}\}\|\tilde{u}_0\|_{C_b^{k+1}(\Omega_1\cup\Omega_2)}.\]
	Choosing $v^h = I_h \tilde{u}_0 - u^h$ in \eqref{orthogonality} and  using the coercivity and continuity of $a(\cdot, \cdot)$ leads to
	\[
	\|u^* - u^h\|_\delta  \leq \|u^* - I_h \tilde{u}_0\|_{\delta}.
	\]
By the triangle inequality and Lemma \ref{lem:energy_bound},
	\begin{align*}
		\| u^* -  I_h \tilde{u}_0 \|_\delta  &\leq \| u^* - u \|_{\delta}  + \| u - \tilde{u}_0 \|_{\delta} + \| \tilde{u}_0 - I_h \tilde{u}_0 \|_\delta\\
		& \leq C \max \{ \delta_1^{-1}, \delta_2^{-1} \} h^{k+1}+ C\max\{\delta_1, \delta_2\}. 
	\end{align*}
For $h\leq \min\{\delta_1, \delta_2\}$, the terms involving $h^{k+1}$ dominate, yielding 
\begin{align*}
		\| u^* -  u^h \|_\delta  \leq C \max \{ \delta_1^{-1}, \delta_2^{-1} \} h^{k+1}. 
	\end{align*}
Thus,
	\[\|u-u^h\|_{\delta}\leq \|u-u^*\|_{\delta} +\|u^*-u^h\|_{\delta} \leq C \max \{ \delta_1^{-1}, \delta_2^{-1} \} h^{k+1},
	\]
	which proves \eqref{error-delta}. 
	For the $L^2$ estimate, since $u^* - u^h \in V^0$, Lemma \ref{lem:poincare} gives
	\[
	\| u^* - u^h \|_{L^2(\Omega_1 \cup \Omega_2)} \lesssim \| u^* - u^h \|_\delta \lesssim \max \{ \delta_1^{-1}, \delta_2^{-1} \} h^{k+1}.
	\]
	Applying the triangle inequality again yields
	\[
	\|u-u^h\|_{L^2(\Omega_1 \cup \Omega_2)} \leq \|u-u^*\|_{L^2(\Omega_1 \cup \Omega_2)} + \|u^*-u^h\|_{L^2(\Omega_1 \cup \Omega_2)} \leq C \max \{ \delta_1^{-1}, \delta_2^{-1} \} h^{k+1},
	\]
	which completes  the proof of   \eqref{error-l2}.
	\end{proof}
	
\begin{remark}
	The convergence rates stated in \eqref{error-l2} may be suboptimal, as they rely on the generic inequality \eqref{L^2}.  However, numerical experiments in Section~\ref{sec;num-ex} indicate that the discrete $L_2$ error converges one order higher than the energy error,  suggesting that the theoretical bounds can be sharpened. The current estimates are therefore regarded as a baseline.
\end{remark}

	\subsection{Error estimates between nonlocal discrete solutions and local exact solution}
	Combining the results of Sections \ref{sec;nonlocal-local} and \ref{sec;unfitted-FEM}, we establish the convergence of the nonlocal finite element  solutions to the local limit solution. 
	\begin{theorem}
	\label{thm:discrete_to_local}
Suppose the kernel family $\{\gamma_{\delta_i}\}$ satisfies  \eqref{kernel_properties}, \eqref{kernel_scaling} and \eqref{kernel_bound}. Let $C^{\max(4,k+1)}_b(\Omega_1 \cup \Omega_2)$ be the solution to the local interface problem \eqref{interface-eq}--\eqref{interface-jump}, and let $\tilde{u}_0$ be a $C^4$ extension of $u_0$ to $\Omega \cup \mathcal{I}^D$  such that $\tilde{u}_0|_{\mathcal{I}_i^D} = g_i$. Then for the unfitted FEM approximation $u^h$, there exists constant $C > 0$, independent of $\delta_1, \delta_2$ and $h$, such that
\begin{align}
\|\tilde{u}_0 - u^h\|_\delta &\leq C \max\{\delta_1^{-1}, \delta_2^{-1}\} h^{k+1}, \label{eq:energy_error_local} \\
\|\tilde{u}_0 - u^h\|_{L^2(\Omega_1 \cup \Omega_2)} &\leq C \max\{\delta_1^{-1}, \delta_2^{-1}\} h^{k+1}. \label{eq:L2_error_local}
\end{align}
\end{theorem}

\begin{proof}
Let $u$ denote the exact  nonlocal solution to \eqref{non-interface-eq}--\eqref{non-interface-jump}. Splitting the error gives 
    \begin{equation}\label{eq:error_split}
    \|\tilde{u}_0 - u^h\|_\delta \leq \|\tilde{u}_0 - u\|_\delta + \|u - u^h\|_\delta. 
    \end{equation}
 From Theorem \ref{thm:nonlocal_to_local} and Lemma \ref{lem:energy_bound},
     \begin{equation}
    \|\tilde{u}_0 - u\|_\delta \leq C \max\{\delta_1^{-1}, \delta_2^{-1}\} \|\tilde{u}_0 - u\|_{\infty} \leq C \max\{\delta_1, \delta_2\}.
    \end{equation}
    Theorem \ref{thm:FEM_error} yields
    \begin{equation}
    \|u - u^h\|_\delta \leq C \max\{\delta_1^{-1}, \delta_2^{-1}\} h^{k+1}.
    \end{equation}
   For $h \leq \min(\delta_1, \delta_2)$, the discrete error dominates, yielding \eqref{eq:energy_error_local}. 
   The $L^2$ estimate follows similarly using Lemma \ref{lem:poincare} and noting $\|u_0 - \tilde{u}_0\|_{L^2} = 0$ on $\Omega$.
\qedhere
\end{proof}
	
	\section{Numerical examples}
	\label{sec;num-ex}
	This section presents numerical experiments to validate the theoretical predictions and to evaluate the performance of the unfitted finite-element method for nonlocal elliptic interface problems. Three representative configurations are examined: a solution continuous across the interface (Example~\ref{eg1}), a solution exhibiting an interfacial jump (Example~\ref{eg2}), and a multi-interface configuration (Example~\ref{eg3}).
The computed errors confirm the convergence rates asserted in Theorem~\ref{thm:discrete_to_local} and demonstrate that the method resolves solution discontinuities with high fidelity.

	\begin{example}\label{eg1}
		The domain is taken as  $\Omega=(0,1) $ with the interface point at $ \alpha =\pi/6$.
For nonlocal interaction horizons parameters  $\delta_1$ and $\delta_2$, the kernel functions are selected as
		$\gamma_{\delta_{i}}(x, y) =6\delta_i^{-4}(\delta_i-|x-y|),\; i=1,2.$
		The exact solutions $u_1(x)$ and $u_2(x)$ are constructed as
		$$
		\begin{array}{cc}
			u_1(x) = \begin{cases} 
				0 &  - \delta_1 < x \leq 0, \\
				\sin x & 0 < x \leq \alpha, \\
				3 / 2 - 2 \sin x & \alpha < x \leq \alpha + \delta_1 ,
			\end{cases}
			&
			u_2(x) = \begin{cases} 
				\sin x & \alpha - \delta_2 < x \leq \alpha ,\\
				3 / 2 - 2 \sin x & \alpha < x < 1, \\
				3 / 2 - 2 \sin 1 & 1 \leq x < 1 + \delta_2.
		\end{cases}
		\end{array}
		$$
	The solution is continuous across the interface, satisfying  $u_1(\alpha)=u_2(\alpha)$. 
		
	The left panel of Figure \ref{fig:eg}  presents the numerical solution obtained with polynomial degree $k=2$. The approximation agrees with the exact solution to graphical accuracy, including at the interface.
		 
Table \ref{tab-ex1-2} lists  the errors in $L^2$ and energy norms by fixed $\delta_i$ and varying $h$, showing optimal $k+1$-order convergence for piecewise polynomials of degree $k$.
 Table  \ref{tab-ex1-3} shows the $\delta$-convergence where $\delta = Mh$ is fixed while $h$ is refined. The results indicate $k$-order convergence in the energy norm and  $k+1$-order convergence in the $L^2$ norm, suggesting  potential for improved  theoretical estimates.
				
		\begin{table}[H]
			\centering
			\caption{ Error norms and convergence rates for  $\delta_1 = 1/4$ and $\delta_2 = 1/2$ }
			\begin{tabular}{|l|l@{\hspace{8pt}}l|l@{\hspace{8pt}}l|l@{\hspace{8pt}}l|}
				\hline
				\multicolumn{1}{|l|}{} & \multicolumn{2}{l|}{$k=1$} & \multicolumn{2}{l|}{$k=2$} & \multicolumn{2}{l|}{$k=3$} \\[8pt]
				\hline
				$h$ & $||u - u^h||_{\delta}$  & $||u-u^h||_{L^2}$ &   $||u - u^h||_{\delta}$  & $||u-u^h||_{L^2}$ &   $||u - u^h||_{\delta}$  & $||u-u^h||_{L^2}$ \\[8pt]
				\hline
				$2^{-2}$ &1.50e-02  & 5.00e-03 &5.43e-04 & 1.46e-04 &3.91e-06 & 2.15e-06  \\[8pt]
				\hline
				$2^{-3}$ &4.10e-03[1.87] &1.30e-03[1.94] &7.39e-05[2.88]&1.86e-05[2.97]&2.65e-07[3.89]&1.37e-07[3.96] \\[8pt]
				\hline
				$2^{-4}$ &1.10e-03[1.90]&3.16e-04[2.04]  &9.57e-06[2.95]&2.39e-06[2.96] &1.70e-08[3.96] &8.76e-09[3.97]\\[8pt]
				\hline
				$2^{-5}$ &2.82e-04[1.96]&7.94e-05[1.99]&1.23e-06[2.96]&3.06e-07[2.96]& 1.10e-09[3.96]&5.55e-10[3.98] \\[8pt]
				\hline
			\end{tabular}
			\label{tab-ex1-2}
		\end{table}	
		
		\begin{table}[H]  
			\centering  
			\caption{ Error norms and convergence rates for  $\delta_1 = 2h$ and $\delta_2 = 4h$}  
			\begin{tabular}{|l|l@{\hspace{8pt}}l|l@{\hspace{8pt}}l|l@{\hspace{8pt}}l|}
				\hline
				\multicolumn{1}{|l|}{} & \multicolumn{2}{l|}{$k=1$} & \multicolumn{2}{l|}{$k=2$} & \multicolumn{2}{l|}{$k=3$} \\[8pt]
				\hline
				$h$ & $||u - u^h||_{\delta}$  & $||u - u^h||_{L^2}$ &   $||u - u^h||_{\delta}$  & $||u - u^h||_{L^2}$ &   $||u - u^h||_{\delta}$  & $||u - u^h||_{L^2}$ \\[8pt]
				\hline
				$2^{-3}$ &4.10e-03 & 1.30e-03&7.39e-05 & 1.86e-05 & 2.65e-07 & 1.37e-07\\[8pt]
				\hline
				$2^{-4}$ &1.40e-03[1.55]&3.13e-04[2.05]  &1.89e-05[1.96]&2.38e-06[2.96]&3.17e-08[3.06]  & 8.72e-09[3.98]   \\[8pt]
				\hline
				$2^{-5}$ &5.12e-04[1.45] &7.84e-05[2.00] &4.84e-06[1.96]&3.05e-07[2.96] &3.93e-09[3.01] &5.54e-10[3.98]  \\[8pt]
				\hline
				$2^{-6}$ &1.77e-04[1.53] &1.96e-05[2.00] & 1.22e-06[1.99] &3.84e-08[2.99] & 4.85e-10[3.02] &3.46e-11[4.00]  \\[8pt]
				\hline
			\end{tabular}
			\label{tab-ex1-3}
		\end{table}
\end{example}

\begin{figure}[H]
	\centering
		\includegraphics[width=0.32\linewidth]{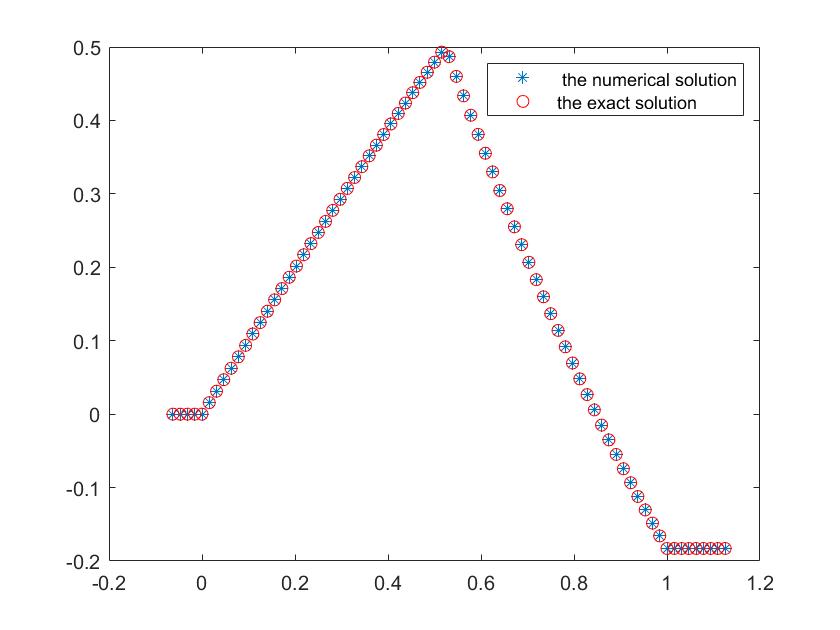} 
		\includegraphics[width=0.32\linewidth]{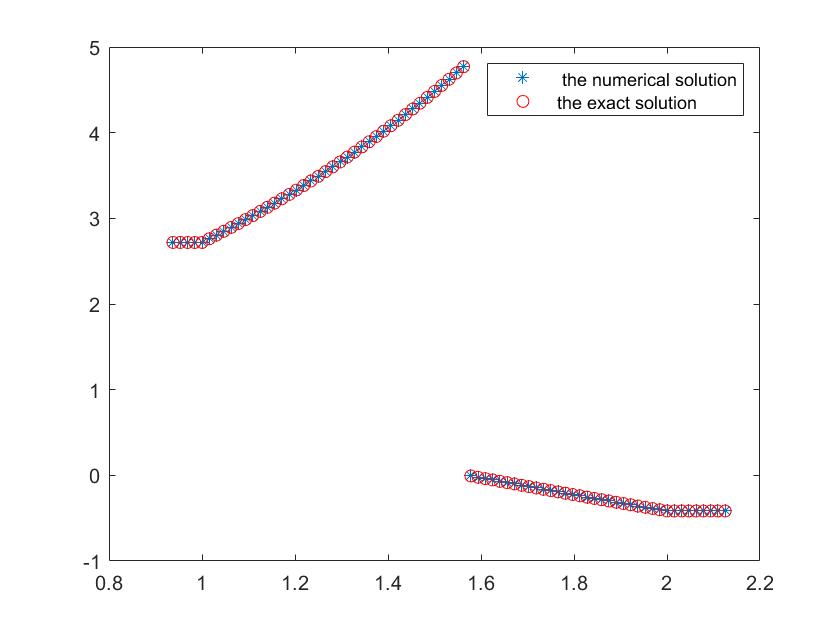}
		\includegraphics[width=0.32\linewidth]{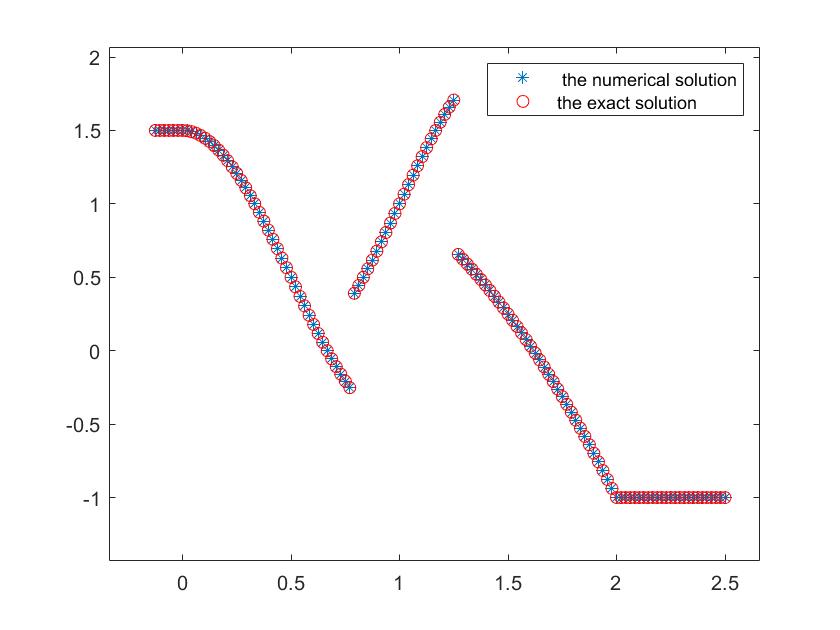}
		\caption{The numerical solution of Example \ref{eg1}-\ref{eg3}.}
		\label{fig:eg}
\end{figure}

	\begin{example}\label{eg2}
	This example tests the case with a solution discontinuity at the interface in $\Omega=(1,2) $ with $ \alpha =\pi/2$.  The kernel functions are  $\gamma_{\delta_{i}}=3\delta_i^{-3}/2, \;i=1,2$.
	The exact solutions with a jump at $\alpha$ are
	$$
	\begin{array}{cc}
		u_1(x) = \begin{cases} 
			e & 1 - \delta_1 < x \leq 1 \\
			e^x & 1 < x \leq \alpha \\
			\cos x & \alpha < x \leq \alpha + \delta_1 ,
		\end{cases}
		&
		u_2(x) = \begin{cases} 
			e^x & \alpha - \delta_2 < x \leq \alpha \\
			\cos x & \alpha < x < 2 \\
			\cos 2 & 2 \leq x < 2 + \delta_2 .
		\end{cases}
	\end{array}
	$$
	
	Tables \ref{tab-ex2-2} and \ref{tab-ex2-3} illustrate that the method accurately captures the jump condition, handling discontinuities without loss of accuracy. Similar to Example \ref{eg1},  the convergence rates match the theoretical expectations, with $L^2$ errors again exhibiting higher-order convergence.  Figure \ref{fig:eg} (middle) compares the numerical solutions for polynomial degrees $k = 2$ with the exact solution, confirming the method's sharp resolution of the discontinuity at the interface.

	\begin{table}[H]  
		\centering  
		\caption{ Error norms and convergence rates for $\delta_1 = 1/4$ and $\delta_2 = 1/2$}  
		\begin{tabular}{|l|l@{\hspace{8pt}}l|l@{\hspace{8pt}}l|l@{\hspace{8pt}}l|}
			\hline
			\multicolumn{1}{|l|}{} & \multicolumn{2}{l|}{$k=1$} & \multicolumn{2}{l|}{$k=2$} & \multicolumn{2}{l|}{$k=3$} \\[8pt]
			\hline
			$h$ & $||u - u^h||_{\delta}$  & $||u - u^h||_{L^2}$ &   $||u - u^h||_{\delta}$  & $||u - u^h||_{L^2}$ &   $||u - u^h||_{\delta}$  & $||u - u^h||_{L^2}$ \\[8pt]
			\hline
			$2^{-2}$ & 4.32e-02 & 1.33e-02 &  1.30e-03 & 3.47e-04 & 1.43e-05 & 6.07e-06  \\[8pt]
			\hline
		$2^{-3}$ & 1.04e-02[2.05] & 3.30e-03[2.01] & 1.65e-04[2.98]& 4.29e-05[3.02] & 8.94e-07[4.00] & 3.73e-07[4.03] \\[8pt]
			\hline
			$2^{-4}$ & 2.80e-03[1.89] & 9.13e-04[1.85] & 2.24e-05[2.88] & 5.85e-06[2.88]  & 6.05e--08[3.88] & 2.55e-08[3.87]\\[8pt]
			\hline
			$2^{-5}$ & 6.92e-04[2.02]  &  2.27e-04[2.01]& 2.80e-06[3.00]  &  7.27e-07[3.01]  & 3.79e-09[4.00]  &  1.58e-09[4.01] \\[8pt]
			\hline
		\end{tabular}
		\label{tab-ex2-2}
	\end{table}		
	
	\begin{table}[H]  
		\centering  
		\caption{ Error norms and convergence rates for $\delta_1 = 2h$ and $\delta_2 = 4h$}  
		\begin{tabular}{|l|l@{\hspace{8pt}}l|l@{\hspace{8pt}}l|l@{\hspace{8pt}}l|}
			\hline
			\multicolumn{1}{|l|}{} & \multicolumn{2}{l|}{$k=1$} & \multicolumn{2}{l|}{$k=2$} & \multicolumn{2}{l|}{$k=3$} \\[8pt]
			\hline
			$h$ & $||u - u^h||_{\delta}$  & $||u - u^h||_{L^2}$ &   $||u - u^h||_{\delta}$  & $||u - u^h||_{L^2}$ &   $||u - u^h||_{\delta}$  & $||u - u^h||_{L^2}$ \\[8pt]
			\hline
			$2^{-3}$ & 1.04e-02 &3.30e-03& 4.41e-08 & 4.29e-05 & 8.94e-07 & 3.73e-07\\[8pt]
			\hline
			$2^{-4}$ & 4.00e-03[1.38] & 9.04e-04[1.87] & 4.51e-05[1.87] & 5.85e-06[2.88] & 1.19e-07[2.92] & 2.55e-08[3.87]   \\[8pt]
			\hline
			$2^{-5}$ & 1.40e-03[1.51]  &  2.24e-04[2.01] & 1.13e-05[1.99] & 7.26e-07[3.01] & 1.47e-08[3.02] & 1.59e-09[4.00]  \\[8pt]
			\hline
			$2^{-6}$ & 4.98e-04[1.49]  & 5.61e-05[2.00] & 2.84e-06[2.00] &  9.03e-08[3.01] & 1.82e-09[3.01]  & 9.78e-11[4.02] \\[8pt]
			\hline
		\end{tabular}
		\label{tab-ex2-3}
	\end{table}
\end{example}

\begin{example}\label{eg3}
	The domain $\Omega=(0,2) $ is considered with interfaces at $ \alpha_1 =\pi / 4$ and $\alpha_2 = 2\pi/5$. For given $\delta_1,\delta_2, \delta_3$, the kernel functions are defined  as 
	\[
	\gamma_{\delta_{i}} =\dfrac{12}{\delta_i^5}(\delta_i^2-|x-y|^2),\qquad i = 1, 2, 3.\] 
	The exact solutions $u_1(x)$ , $u_2(x)$ and $ u_3(x) $ are designed to exhibit jumps at both interfaces
	$$
	\begin{array}{cc}
	u_1(x) = \begin{cases} 
		\frac{1}{2} &  - \delta_1 < x \leq 0, \\
		\frac{1}{2}\cos(\pi x) & 0 < x \leq \alpha_1 \\
		1-	\sin(\pi x) & \alpha_1 < x \leq \alpha_1 + \delta_1 ,
	\end{cases}\\
		u_2(x) = \begin{cases} 
			\frac{1}{2}\cos(\pi x) & \alpha_1 - \delta_2 < x \leq \alpha_1, \\
			1-	\sin(\pi x) & \alpha_1 < x \leq \alpha_2 \\
			x(1 - x)+1 & \alpha_2 < x < \alpha_2 + \delta_2 .
		\end{cases}
		u_3(x) = \begin{cases} 
			1-	\sin(\pi x) & \alpha_2 - \delta_2 < x \leq \alpha_2, \\
			x(1 - x)+1 & \alpha_2 < x < 2 \\
			-1 & 2 \leq x < \alpha_2/2 + \delta_2 .
		\end{cases}
	\end{array}
	$$
	
	Tables \ref{tab-ex3-2} and \ref{tab-ex3-3} present the error norms and convergence rates for fixed horizons and horizon-to-mesh ratios, respectively. The results confirm the theoretical estimates, demonstrating optimal convergence rates in both the energy and $L^2$ norms.  Figure \ref{fig:eg} (right) shows that the method accurately resolves discontinuities at multiple interfaces, further validating its robustness for problems involving complex interfacial geometries.

\begin{table}[H]  
  \centering  
  \caption{ Errors norms and convergence rates for $\delta_1 = 1/8$, $\delta_2=1/4$ and $\delta_3 = 1/2$}   
  \begin{tabular}{|l|l@{\hspace{8pt}}l|l@{\hspace{8pt}}l|l@{\hspace{8pt}}l|}
   \hline
   \multicolumn{1}{|l|}{} & \multicolumn{2}{l|}{$k=1$} & \multicolumn{2}{l|}{$k=2$} & \multicolumn{2}{l|}{$k=3$} \\[8pt]
   \hline
   $h$ & $||u - u^h||_{\delta}$  & $||u - u^h||_{L^2}$ &   $||u - u^h||_{\delta}$  & $||u - u^h||_{L^2}$ &   $||u - u^h||_{\delta}$  & $||u - u^h||_{L^2}$ \\[8pt]
   \hline
   $2^{-3}$ & 8.00e-02 & 7.50e-03 &  7.00e-03 & 4.61e-04 & 8.99e-05 & 8.39e-06  \\[8pt]
   \hline
   $2^{-4}$ & 2.03e-02[1.98] & 2.00e-03[1.91] & 8.67e-04[3.01]& 5.82e-05[2.99] & 5.50e-06[4.03] &  5.35e-07[3.97] \\[8pt]
   \hline
   $2^{-5}$ & 5.10e-03[1.99] & 5.06e-04[1.98] & 1.08e-04[3.01] & 7.28e-06[3.00]  & 3.39e--07[4.02] & 3.39e-08[3.98]\\[8pt]
   \hline
   $2^{-6}$ & 1.20e-03[2.09]  &   1.39e-04[1.86]& 1.33e-05[3.02]  &  9.04e-07[3.01]  & 2.10e-08[4.01]  &  2.12e-09[4.00] \\[8pt]
   \hline
  \end{tabular}
  \label{tab-ex3-2}
 \end{table}  
 
 \begin{table}[H]  
  \centering  
  \caption{ Errors norms and convergence rates for $\delta_1 = h$, $\delta_2 = 2h$ and $\delta_3 = 3h$}  
  \begin{tabular}{|l|l@{\hspace{8pt}}l|l@{\hspace{8pt}}l|l@{\hspace{8pt}}l|}
   \hline
   \multicolumn{1}{|l|}{} & \multicolumn{2}{l|}{$k=1$} & \multicolumn{2}{l|}{$k=2$} & \multicolumn{2}{l|}{$k=3$} \\[8pt]
   \hline
   $h$ & $||u - u^h||_{\delta}$  & $||u - u^h||_{L^2}$ &   $||u - u^h||_{\delta}$  & $||u - u^h||_{L^2}$ &   $||u - u^h||_{\delta}$  & $||u - u^h||_{L^2}$ \\[8pt]
   \hline
   $2^{-3}$ & 8.02e-02 &7.40e-03& 7.00e-03 & 4.61e-04 & 9.00e-05 & 8.38e-06\\[8pt]
   \hline
   $2^{-4}$ & 2.75e-02[1.54] & 1.90e-03[1.96] & 1.80e-03[1.96] & 5.96e-05[2.95] & 1.12e-05[3.01] & 5.36e-07[3.97]   \\[8pt]
   \hline
   $2^{-5}$ & 9.80e-03[1.49]  &  4.96e-04[1.94] & 4.54e-04[1.99] & 7.89e-06[2.92] & 1.42e-06[2.98] & 3.49e-08[3.94]  \\[8pt]
   \hline
   $2^{-6}$ & 3.50e-03[1.49]  & 1.27e-04[1.97] & 1.14e-04[1.99] &  9.71e-07[3.02] & 1.77e-07[3.00]  & 2.40e-09[3.86] \\[8pt]
   \hline
  \end{tabular}
  \label{tab-ex3-3}
 \end{table}
 \end{example}

	The numerical results validate the theoretical error estimates even for solutions discontinuous across the interface.  In all case, the discrete $L^2$ error converges at an order higher than predicted, indicating that the current theoretical bounds can be improved.  The method's robustness with respect to interfacial jumps underscores its applicability to practical problems involving material discontinuities or sharp phase boundaries.	
	
\section{Conclusion}
This paper has introduced an AC unfitted finite element framework for one-dimensional nonlocal elliptic interface problems. The analysis establishes second-order convergence of the nonlocal solution to the local limit as the horizon parameters tend to zero. By incorporating nonlocal jump conditions directly into the weak formulation via a Nitsche-type approach, the method accommodates solution discontinuities without requiring mesh conformity to the interface. Optimal convergence rates of order $(k+1)$ in both the energy and $L^2$ norms are rigorously proven for piecewise polynomial approximations of degree $k$.

Numerical experiments demonstrate the robustness of the scheme for both continuous and discontinuous interfacial solutions, confirming the theoretical findings.
The proposed asymptotically compatible unfitted framework offers  a practical and efficient  foundation for addressing more complex problems in computational mechanics.  
A natural and critical extension of this work is to higher spatial dimensions, which would significantly broaden its applicability to realistic engineering problems. Future research will also focus on time-dependent nonlocal interface problems, such as dynamic fracture and elastodynamics with evolving micro-cracks. The mathematical analysis and numerical tools developed in this one-dimensional study serve as a cornerstone for these ambitious extensions.

\section{Appendix}	
	This appendix provides detailed proofs of the key lemmas and theorems.
Section~\ref{app:flux_proof} contains the proof of  Lemma~\ref{lemma:flux_limit}, which establishes the local limit of the interface flux operator. 
Section~\ref{app:norm_proof} presents the proofs of  Lemmas~\ref{lem:energy_bound} and~\ref{lem:poincare}, which are essential to the error analysis of the unfitted finite element scheme.
	
	\subsection{Proof of Lemma~\ref{lemma:flux_limit} }\label{app:flux_proof}
	The nonlocal flux operator $\mathcal{F}(\hat{u}_1, \hat{u}_2)$  is analyzed by decomposing its action as follows
	\begin{align*}
		(\mathcal{F}(\hat{u}_1, \hat{u}_2),v)_{\mathcal{I}_2^J} + (\mathcal{F}(\hat{u}_1, \hat{u}_2),v)_{\mathcal{I}_1^J} = \mathcal{I}_1+\mathcal{I}_2+\mathcal{I}_3 + \mathcal{I}_4+\mathcal{I}_5,
	\end{align*}
	where the terms are defined by 
	\begin{align*}
		\mathcal{I}_1=&\int_{\alpha - \delta_2}^{\alpha}\int_{\alpha +
			\delta_1}^{\alpha + \delta_2}\left[\hat{u}_2(x)-\hat{u}_2(y)\right]\gamma_{\delta_2}(x,y)v(x)dydx,  \; 
		\mathcal{I}_2=\frac{1}{2}\int_{\alpha - \delta_2}^{\alpha}\int_{\alpha}^{\alpha + \delta_1}\left[\hat{u}_1(x)-\hat{u}_1(y)\right]\gamma_{\delta_2}(x,y)v(x)dydx,\\[4pt]
		\mathcal{I}_3=&-\frac{1}{2}\int_{\alpha - \delta_2}^{\alpha}\int_{\alpha}^{\alpha + \delta_1}\left[\hat{u}_1(x)-\hat{u}_1(y)\right]\gamma_{\delta_1}(x,y)v(x)dydx, \; 
		\mathcal{I}_4 = \frac{1}{2}\int_{\alpha}^{\alpha + \delta_1}\int_{\alpha - \delta_2}^{\alpha}\left[\hat{u}_2(x)-\hat{u}_2(y)\right]\gamma_{\delta_1}(x,y)v(x)dydx,\\[4pt]
		\mathcal{I}_5=&-\frac{1}{2}\int_{\alpha}^{\alpha + \delta_1}\int_{\alpha- \delta_2}^{\alpha}\left[\hat{u}_2(x)-\hat{u}_2(y)\right]\gamma_{\delta_2}(x,y) v(x)dydx.
	\end{align*}
	
	For $\mathcal{I}_1$, the coordinate transformations 
	$x = \alpha - t\delta_2,$
	$y = \alpha + s\delta_2 $ with  $t \in [0,1], \;s \in [\delta_1/\delta_2,1]$ yield
	\[
		\mathcal{I}_1 
		 = \frac{1}{\delta_2} \int_{t=0}^1 \int_{s=\delta_1/\delta_2}^1 [\hat{u}_2(\alpha-t\delta_2)-\hat{u}_2(\alpha+s\delta_2)] \gamma_2(t+s) v(\alpha-t\delta_2) \, ds \, dt.	\]
		 
	  Taylor expansions about $\alpha$ give
	\begin{equation*}
		\begin{split}
			\hat{u}_2(\alpha-t\delta_2)& = \hat{u}_2(\alpha) - \hat{u}_2'(\alpha)t\delta_2 + \frac{1}{2}\hat{u}_2''(\alpha)t^2\delta_2^2 +\mathcal{O}(\delta_2^3),\\
			\hat{u}_2(\alpha+s\delta_2) &= \hat{u}_2(\alpha) + \hat{u}_2'(\alpha)s\delta_2 + \frac{1}{2}\hat{u}_2''(\alpha)s^2\delta_2^2+\mathcal{O}(\delta_2^3),\\
			v(\alpha-t\delta_2) &= v(\alpha)+ \mathcal{O}(\delta_2).
		\end{split} 
	\end{equation*}
	Substituting into $\mathcal{I}_1 $ leads to
	\[
		\mathcal{I}_1 
		= -\hat{u}_2'(\alpha)v(\alpha)\mathcal{I}_1^1
		+\frac{1}{2}\delta_2\hat{u}_2''(\alpha) v(\alpha)\mathcal{I}_1^2+ O(\delta_2^2),\\
	\]
	where 
\[
\begin{split}
\mathcal{I}_1^1 &=  \int_{t=0}^1 \int_{s=\delta_1/\delta_2}^1 (t+s)\gamma_2(t+s) \, ds \, dt,\qquad
\mathcal{I}_1^2  = \int_{t=0}^1 \int_{s=\delta_1/\delta_2}^1 (t^2-s^2)\gamma_2(t+s) \, ds \, dt.
\end{split}
\]
Similar expressions apply to $\mathcal{I}_2$-$\mathcal{I}_5$.   Combining the terms involving $\gamma_1$-kernel gives
\begin{equation*}
		\begin{split}
			\mathcal{I}_3  + \mathcal{I}_4
			=&\hat{u}_1'(\alpha)v(\alpha)\mathcal{I}_{\sigma_1}^1 +  \frac{1}{2}\delta_1 \hat{u}_1''(\alpha)v(\alpha) \mathcal{I}_{\sigma_1}^2 +O(\delta_1^2),
		\end{split}
	\end{equation*}
	with 
	\[
	\begin{split}
	\mathcal{I}_{\sigma_1}^1 &= \int_{s=0}^{1}\int_{t=0}^{\delta_2/\delta_1}(s + t)\gamma_1(s + t)\,dt\,ds,\qquad
	\mathcal{I}_{\sigma_1}^2 = \int_{t=0}^{1}\int_{z=t}^{t + \delta_2/\delta_1}(2zt - z^{2})\gamma_1(z)dzdt.
	\end{split}
	\]
Similarly, combining the $\gamma_2$-terms yields
	\begin{align*}
		\mathcal{I}_1+\mathcal{I}_2+\mathcal{I}_5 =
		 -\hat{u}_2'(\alpha)v(\alpha)\mathcal{I}_{\sigma_2}^1	 + \frac{1}{2}\delta_2 \hat{u}''_2(\alpha) v(\alpha)\mathcal{I}_{\sigma_2}^2+O(\delta_2^2),
	\end{align*}
where
	\[\mathcal{I}_{\sigma_2}^1 =  \int_0^1 \int_0^1 (t+s)\gamma_2(t+s)\,ds\,dt, \qquad
	\mathcal{I}_{\sigma_2}^2  = \int_{t=0}^1 \int_{z=t}^{t+1} (2zt - z^2)\gamma_2(z)\,dz\,dt.\]
Direct calculation shows that $\mathcal{I}_{\sigma_1}^2=\mathcal{I}_{\sigma_2}^2 =0$.  For $\mathcal{I}^1_{\sigma_1}$,
	\[
	\begin{split}
		\mathcal{I}^1_{\sigma_1} &= \int_{s=0}^{1} \int_{z=s}^{s + \delta_2/\delta_1} z\gamma_1(z) \, dz \, ds,\\
		&= \int_{z=0}^{1} \int_{s=0}^{z} z \gamma_1(z) \, ds \, dz+ \int_{z=1}^{\delta_2/\delta_1} \int_{s=0}^{1} z \gamma_1(z) \, ds \, dz =\int_{0}^{1} z^2 \gamma_1(z) \, dz.
	\end{split}
	\]
	and for $ \mathcal{I}^1_{\sigma_2}$, 	
	\begin{align*}
		\mathcal{I}^1_{\sigma_2} 
		&= \int_{x=0}^{2} \left( \int_{t=\max(0,x-1)}^{\min(1,x)} x \gamma_2(x) \, dt \right) dx \\
		&= \int_{x=0}^{1} x^2 \gamma_2(x) \, dx + \int_{x=1}^{2} x(2 - x) \gamma_2(x) \, dx  = \int_{0}^{1} x^2 \gamma_2(x) \, dx. 
	\end{align*}
	Recalling the moment condition \eqref{kernel_scaling}, it follows that $\mathcal{I}^1_{\sigma_i} = 2\sigma_i $.
	Combining all terms gives the desired result
	\begin{align*}
		\mathcal{F}(\hat{u}_1, \hat{u}_2)=\sigma_1\hat{u}_1^\prime(\alpha)-\sigma_2\hat{u}_2^\prime(\alpha)+\mathcal{O}(\max(\delta_1^2, \delta_2^2)).
	\end{align*}
	This completes the proof of Lemma \ref{lemma:flux_limit}.
	
	\subsection{Proofs of Lemmas \ref{lem:energy_bound} and \ref{lem:poincare}}	\label{app:norm_proof}
	
	\noindent\textbf{Proof of Lemmas \ref{lem:energy_bound}.}
		Each component of the energy norm is estimated separately.  For  \(\lVert v\rVert_{\delta_i}^{2}\), 
		\[
		\begin{split}
			\|v\|_{\delta_i}^2 &\leq \int_{\Omega_i\cup\mathcal{I}_i^D} \int_{\Omega_i\cup\mathcal{I}_i^D} (v_i^2(x)+v_i^2(y)) \gamma_{\delta_i}(x,y) \, dy \, dx
			= 2\int_{\Omega_i\cup\mathcal{I}_i^D} v_i^2(x) \left(\int_{\Omega_i\cup\mathcal{I}_i^D} \gamma_{\delta_i}(x,y) \, dy\right) dx. 
		\end{split}
		\]
Using assumption \eqref{kernel_bound} and the boundedness of  $v_i$ 
		\[
		\begin{split}
			\|v\|_{\delta_i}^2 &\leq 2G(\gamma_{\delta_i})\left(|\Omega_i|\|v_i\|_{\infty, \Omega_i}^2 + |\mathcal{I}_i^D|\|v_i\|_{\infty, \mathcal{I}_i^D}^2\right)
			\lesssim \delta_i^{-2}\left(\|v_i\|_{\infty, \Omega_i}^2 + \|v_i\|_{\infty, \mathcal{I}_i^D}^2\right).
		\end{split}
		\]
The $\Gamma$-norm terms are bounded similarly.  For example		\[
		\begin{split}
			&\frac{1}{2} \int_{\mathcal{I}_2^J} \int_{\mathcal{I}_1^J} (v_1(x)-v_1(y))^2 \gamma_{\delta_1}(x,y) \, dy \, dx
			\lesssim \delta_1^{-2} \left(|\mathcal{I}_2^J|\|v_1\|_{\infty, \mathcal{I}_2^J}^2 + |\mathcal{I}_1^J|\|v_1\|_{\infty, \mathcal{I}_1^J}^2\right).
		\end{split}
		\]
The remaining terms are handled analogously, leading to		\[
		\|v\|_\Gamma^2 \lesssim \max\{\delta_1^{-1}, \delta_2^{-1}\} \left(\|v\|_{\infty, \mathcal{I}_1^J}^2 + \|v\|_{\infty, \mathcal{I}_2^J}^2 \right).
		\]
Combining these estimates yields
		\[
		\|v\|_\delta^2 \lesssim \max\{\delta_1^{-2}, \delta_2^{-2}\} \left(\|v\|_{\infty, \Omega_1\cup\Omega_2}^2 + \|v\|_{\infty, \mathcal{I}_1^D\cup\mathcal{I}_2^D}^2\right).
		\]
		Taking square roots completes the proof.

	\noindent\textbf{Proof of Lemmas \ref{lem:poincare}.}
		For each subdomain $\Omega_i$ ($i=1,2$), the nonlocal Poincaré inequality \cite[Theorem 3.34]{MR3938295} implies
		\begin{equation}\label{eq:poincare}
			\|v_i\|_{L^2(\Omega_i)}^2 \leq C_{p,i}\|v_i\|_{\delta_i}^2, \quad  {\rm where}\;C_{p,i} = \frac{|\Omega_i|^2}{2\sigma_i}.
		\end{equation}
		Since $v = 0$ on ${\mathcal{I}_i^D}$,  it follows that 
				$ \|v\|_{\delta_i}^2 \leq \|v\|_\delta^2. $
		Summing over subdomains 
		\begin{align*}
			\|v\|_{L^2(\Omega_1 \cup \Omega_2)}^2 &= \|v_1\|_{L^2(\Omega_1)}^2 + \|v_2\|_{L^2(\Omega_2)}^2 
			\leq (C_{p,1} + C_{p,2})\|v\|_\delta^2.
		\end{align*}
		The lower bound follows by setting  $C_1 = (C_{p,1} + C_{p,2})^{-1/2}$. For the upper bound   
		\begin{align*}
			\|v\|_{\delta_i}^2
			\leq2\int_{\Omega_i\cup\mathcal{I}_i^D}v_1^2(x)\left(\int_{\Omega_i\cup\mathcal{I}_i^D}\gamma_{\delta_i}(x,y)\,dy\right)dx
			\leq C_{l,i} \delta_i^{-2}\|v\|_{L^2(\Omega_i)}^2.
		\end{align*}
		Similarly, for the $\Gamma$-norm
		  \[\|v\|_{\Gamma}^2\leq  C_{l,3} \max\{\delta_1^{-2},\delta_2^{-2}\} (\|v\|_{L^2(\mathcal{I}_1^J)}^2 + \|v\|_{L^2(
				\mathcal{I}_2^J)}^2).\]
		Combining these results gives 
		\[
		\|v\|_{\delta}\leq C_2 \max\{\delta_1^{-1},\delta_2^{-1}\}\|v\|_{L^2(\Omega_1\cup\Omega_2)},
		\]
		which completes the proof.

\section*{Acknowledgments}
H. Dong was partially supported by the National Natural Science Foundation of China (Grant No. 12001193), the Scientific Research Fund of Hunan Provincial Education Department (Grant No.24A0052), and the Hunan provincial national science foundation of China (No. 2024JJ6298).  Z. Xie was supported by the National Natural Science Foundation of China (Grant No. 12171148, 52331002),  the Major Program of Xiangjiang Laboratory ( Grant No. 22XJ01013).  J. Zhang was partially supported by National Natural Science Foundation of China ( Grant No. 12571438) and the Fundamental Research Funds for the Central Universities (Grant No. 2042021kf0050).

	\bibliographystyle{plain}      

\end{document}